\documentclass{amsart}

\title{The $\Sigma_1$-definable universal finite sequence}

\author{Joel David Hamkins}
 \address[Joel David Hamkins]
          {Professor of Logic, University of Oxford, and
           Sir Peter Strawson Fellow, University College, High Street, Oxford OX1 4BH, UK}
 \email{joeldavid.hamkins@philosophy.ox.ac.uk}
 \urladdr{http://jdh.hamkins.org}

\DeclareRobustCommand{\okina}{%
  \raisebox{\dimexpr\fontcharht\font`A-\height}{%
    \scalebox{0.8}{`}%
  }%
}
\author{Kameryn J. Williams}
\address[Kameryn J. Williams]{
University of Hawai\okina{}i at M\=anoa \\
Department of Mathematics \\
2565 McCarthy Mall, Keller 401A \\
Honolulu, HI  96822 \\
USA}
\email{kamerynw@hawaii.edu}
\urladdr{http://kamerynjw.net}

\thanks{We thank the anonymous referee for their numerous helpful comments.\\
Commentary can be made about this article on the first author's blog at \href{http://jdh.hamkins.org/the-universal-finite-sequence}{http://jdh.hamkins.org/the-universal-finite-sequence}.}

\usepackage{latexsym,amsfonts,amsmath,amssymb,mathrsfs}
\usepackage{relsize}
\usepackage[rgb,dvipsnames]{xcolor}
\usepackage{tikz}
\usetikzlibrary{arrows,arrows.meta,petri,topaths,positioning,shapes,shapes.misc,patterns,calc,decorations.pathreplacing,hobby}
\usepackage[shortlabels]{enumitem} 
\usepackage[pdfauthor={Joel David Hamkins and Kameryn J Williams},
    pdftitle={The Sigma1-definable universal finite sequence},
    hidelinks]{hyperref}

\newtheorem*{main}{Main~Theorem}
\newtheorem{theorem}{Theorem}

\newtheorem*{maintheorem*}{Main Theorem}
\newtheorem*{maintheorems*}{Main Theorems}
\newtheorem{corollary}[theorem]{Corollary}
\newtheorem*{corollary*}{Corollary}
\newtheorem*{corollaries*}{Corollaries}

\newtheorem{lemma}[theorem]{Lemma}

\newtheorem{question}[theorem]{Question}
\newtheorem*{question*}{Question}

\newtheorem*{questions*}{Questions}
\newtheorem*{mainquestion*}{Main Question} 
\newtheorem*{openquestion*}{Open Question} 

\newtheorem{definition}[theorem]{Definition}

\newcommand{\QED}{\end{proof}}

\def\proclaim[#1]{{\bf #1}}
\def\BF#1.{{\bf #1.}}

\def\says#1:#2\par{\item[#1] #2\par}

\newcommand{\Godel}{G\"odel}

\newcommand{\Levy}{L\'{e}vy}
\newcommand{\Lowenheim}{L\"owenheim}

\makeatletter
\newcommand{\dotminus}{\mathbin{\text{\@dotminus}}}
\newcommand{\@dotminus}{%
  \ooalign{\hidewidth\raise1ex\hbox{.}\hidewidth\cr$\m@th-$\cr}%
}
\makeatother

\newcommand{\of}{\subseteq}

\newcommand{\restrict}{\upharpoonright} 
\newcommand{\satisfies}{\models}

\DeclareMathOperator{\possible}{\text{\tikz[scale=.6ex/1cm,baseline=-.6ex,rotate=45,line width=.1ex]{\draw (-1,-1) rectangle (1,1);}}}
\DeclareMathOperator{\necessary}{\text{\tikz[scale=.6ex/1cm,baseline=-.6ex,line width=.1ex]{\draw (-1,-1) rectangle (1,1);}}}

\newcommand{\theoryf}[1]{{\rm #1}}

\newcommand{\smalllt}{\mathrel{\mathchoice{\raise2pt\hbox{$\scriptstyle<$}}{\raise1pt\hbox{$\scriptstyle<$}}{\raise0pt\hbox{$\scriptscriptstyle<$}}{\scriptscriptstyle<}}}
\newcommand{\smallleq}{\mathrel{\mathchoice{\raise2pt\hbox{$\scriptstyle\leq$}}{\raise1pt\hbox{$\scriptstyle\leq$}}{\raise1pt\hbox{$\scriptscriptstyle\leq$}}{\scriptscriptstyle\leq}}}

\newcommand{\boolval}[1]{\mathopen{\lbrack\!\lbrack}\,#1\,\mathclose{\rbrack\!\rbrack}}
\def\[#1]{\boolval{#1}}
\newbox\gnBoxA
\newdimen\gnCornerHgt
\setbox\gnBoxA=\hbox{\tiny$\ulcorner$}
\global\gnCornerHgt=\ht\gnBoxA
\newdimen\gnArgHgt
\def\gcode #1{%
\setbox\gnBoxA=\hbox{$#1$}%
\gnArgHgt=\ht\gnBoxA%
\ifnum     \gnArgHgt<\gnCornerHgt \gnArgHgt=0pt%
\else \advance \gnArgHgt by -\gnCornerHgt%
\fi \raise\gnArgHgt\hbox{\tiny$\ulcorner$} \box\gnBoxA %
\raise\gnArgHgt\hbox{\tiny$\urcorner$}}
\newcommand{\UnderTilde}[1]{{\setbox1=\hbox{$#1$}\baselineskip=0pt\vtop{\hbox{$#1$}\hbox to\wd1{\hfil$\sim$\hfil}}}{}}
\newcommand{\Undertilde}[1]{{\setbox1=\hbox{$#1$}\baselineskip=0pt\vtop{\hbox{$#1$}\hbox to\wd1{\hfil$\scriptstyle\sim$\hfil}}}{}}
\newcommand{\undertilde}[1]{{\setbox1=\hbox{$#1$}\baselineskip=0pt\vtop{\hbox{$#1$}\hbox to\wd1{\hfil$\scriptscriptstyle\sim$\hfil}}}{}}
\newcommand{\UnderdTilde}[1]{{\setbox1=\hbox{$#1$}\baselineskip=0pt\vtop{\hbox{$#1$}\hbox to\wd1{\hfil$\approx$\hfil}}}{}}
\newcommand{\Underdtilde}[1]{{\setbox1=\hbox{$#1$}\baselineskip=0pt\vtop{\hbox{$#1$}\hbox to\wd1{\hfil\scriptsize$\approx$\hfil}}}{}}

\renewcommand{\implies}{\mathrel{\rightarrow}}

\renewcommand{\iff}{\mathrel{\leftrightarrow}}

\def\<#1>{\left\langle#1\right\rangle}

\newcommand{\Ord}{\mathord{{\rm Ord}}}

\newcommand\ACA{{\rm ACA}}

\newcommand{\ZFC}{{\rm ZFC}}
\newcommand{\ZF}{{\rm ZF}}

\newcommand{\ZFCm}{\ZFC^-}

\newcommand{\KP}{{\rm KP}}

\newcommand{\cell}[1]{\boxit{\hbox to 17pt{\strut\hfil$#1$\hfil}}}
\newcommand{\head}[2]{\lower2pt\vbox{\hbox{\strut\footnotesize\it\hskip3pt#2}\boxit{\cell#1}}}
\newcommand{\boxit}[1]{\setbox4=\hbox{\kern2pt#1\kern2pt}\hbox{\vrule\vbox{\hrule\kern2pt\box4\kern2pt\hrule}\vrule}}
\newcommand{\Col}[3]{\hbox{\vbox{\baselineskip=0pt\parskip=0pt\cell#1\cell#2\cell#3}}}
\newcommand{\tapenames}{\raise 5pt\vbox to .7in{\hbox to .8in{\it\hfill input: \strut}\vfill\hbox to
.8in{\it\hfill scratch: \strut}\vfill\hbox to .8in{\it\hfill output: \strut}}}
\newcommand{\Head}[4]{\lower2pt\vbox{\hbox to25pt{\strut\footnotesize\it\hfill#4\hfill}\boxit{\Col#1#2#3}}}
\newcommand{\Dots}{\raise 5pt\vbox to .7in{\hbox{\ $\cdots$\strut}\vfill\hbox{\ $\cdots$\strut}\vfill\hbox{\
$\cdots$\strut}}}
\newcommand{\ZFm}{\ZF^{-}}
\newcommand\ZFbar{\overline\ZF}

\begin{document}

\begin{abstract}
We introduce the $\Sigma_1$-definable universal finite sequence and prove that it exhibits the universal extension property amongst the countable models of set theory under end-extension. That is, (i) the sequence is $\Sigma_1$-definable and provably finite; (ii) the sequence is empty in transitive models; and (iii) if $M$ is a countable model of set theory in which the sequence is $s$ and $t$ is any finite extension of $s$ in this model, then there is an end-extension of $M$ to a model in which the sequence is $t$. Our proof method grows out of a new infinitary-logic-free proof of the Barwise extension theorem, by which any countable model of set theory is end-extended to a model of $V=L$ or indeed any theory true in a suitable submodel of the original model. The main theorem settles the modal logic of end-extensional potentialism, showing that the potentialist validities of the models of set theory under end-extensions are exactly the assertions of \theoryf{S4}. Finally, we introduce the end-extensional maximality principle, which asserts that every possibly necessary sentence is already true, and show that every countable model extends to a model satisfying it.
\end{abstract}

\maketitle

\section{Introduction}\label{Section.Introduction}

We provide a $\Sigma_1$-definition in set theory of a finite sequence and prove that it exhibits a universal extension property amongst the countable models of set theory.

\begin{main}
For any computably axiomatizable theory $\ZFbar$ extending \ZF, there is a $\Sigma_1$-definable finite sequence
$$a_0,\ a_1,\ \ldots,\ a_n$$
with the following properties:
\begin{enumerate}
\item $\ZF$ proves that the sequence is finite.
\item In any transitive model $M$ of $\ZFbar$, the sequence is empty.
\item If $M$ is a countable model of $\ZFbar$ in which the sequence is $s$ and $t \in M$ is any finite sequence (in the sense of $M$) extending $s$, then there is a covering end-extension of $M$ to a model $N \models \ZFbar$ in which the sequence is exactly $t$, and in which every set of $M$ is countable.
\item Indeed, for statements (2) and (3), it suffices merely that $M\satisfies\ZF$ and end-extends a submodel $W\satisfies \ZFbar$ of height at least $(\omega_1^L)^M$.
\end{enumerate}
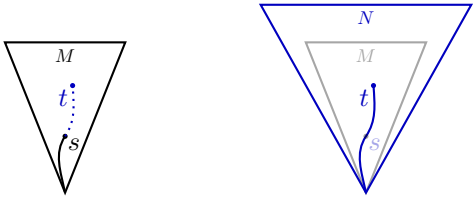
$$\begin{tikzpicture}[xscale=.8,scale=.5]
 \draw[thick] (0,0) -- (-2,4) -- node[below,scale=.7] {$M$} (2,4) -- cycle;
 \draw[thick] (0,0) to[out=105,in=-130] (0,1.5) node[circle,fill=black,scale=.2,label={[label distance=-5pt]below right:$s$}] {};
 \draw[thick,dotted,blue!75!black] (0,1.5) to[out=50,in=-85] (.25,2.85) node[circle,fill=blue!75!black,scale=.2,label={[label distance=-4pt]below left:$t$}] {};
\begin{scope}[shift={(10,0)}]
 \draw[thick,opacity=.35] (0,0) -- (-2,4) -- node[below,scale=.7] {$M$} (2,4) -- cycle;
 \draw[thick,blue!75!black] (0,0) -- (-3.5,5) -- node[below,scale=.7] {$N$} (3.5,5) -- cycle;
 \draw[thick,blue!75!black] (0,0) to[out=105,in=-130] (0,1.5) node[circle,fill opacity=.35, fill=black,scale=.2,label={[opacity=.35,label distance=-5pt]below right:$s$}] {} to[out=50,in=-85] (.25,2.85) node[circle,fill=blue!75!black,scale=.2,label={[label distance=-4pt]below left:$t$}] {};
\end{scope}
\end{tikzpicture}$$
\end{main}

The $\Sigma_1$-definable universal finite sequence is a sister to the universal finite set of Hamkins and Woodin~\cite{HamkinsWoodin:The-universal-finite-set}, which provided a $\Sigma_2$-definable finite set having a universal extension property with respect to the models of set theory under rank-extensions (also known as top-extensions), namely, those for which the new sets of the extension all have higher rank than the old sets; thus, the smaller model is determined by a cut in the $V_\alpha$-hierarchy of the larger model. The Hamkins--Woodin universal set naturally used a $\Sigma_2$-definition, as it is precisely the $\Sigma_2$-assertions that are witnessed in the rank initial segments $V_\alpha$ of the universe.

In the main result of this article, in contrast, we seek universality with respect to arbitrary end-extensions of models of set theory, and since it is precisely the $\Sigma_1$-assertions that are witnessed in $\in$-initial segments of the universe, we seek accordingly a $\Sigma_1$-definable universal sequence. In short, we seek to undertake a $\Sigma_1$ analogue of the $\Sigma_2$ analysis of~\cite{HamkinsWoodin:The-universal-finite-set}, working with end-extensions of models of set theory rather than rank extensions.

Both the $\Sigma_1$-definable universal finite sequence of this article and the $\Sigma_2$-definable universal finite set of~\cite{HamkinsWoodin:The-universal-finite-set} should be seen as set-theoretic analogues of Woodin's universal algorithm, a Turing machine program that provably enumerates a finite sequence of numbers exhibiting a universal extension property with respect to the models of arithmetic under end-extensions~\cite{Woodin2011:A-potential-subtlety-concerning-the-distinction-between-determinism-and-nondeterminism, Hamkins:The-modal-logic-of-arithmetic-potentialism, BlanckEnayat2017:Marginalia-on-a-theorem-of-Woodin, Blanck2017:Dissertation:Contributions-to-the-metamathematics-of-arithmetic}.

To clarify terms, an \emph{end-extension} of a model of set theory $\<M,\in^M>$ is another model $\<N,\in^N>$, such that the first is a $\in$-initial substructure of the second. That is, $M\of N$ and $\mathord{\in^M}=\mathord{\in^N\restrict M}$, but further, the new model does not add new elements to old sets; more precisely: if $a\in^N b\in M$, then $a\in M$ and hence $a\in^M b$. Such an end-extension is a \emph{covering} end-extension if the larger model has a covering set, an object $m\in N$ such that $a\in^N m$ for all $a\in M$. In this article, unless we say otherwise all extensions will be proper extensions. That is, we wish to exclude the trivial extension $M \subseteq M$.

Set theory, of course, overflows with instances of end-extensions and covering end-extensions. The rank-initial segments $V_\alpha$ end-extend to their higher instances $V_\beta$, which also cover them, when $\alpha<\beta$; similarly, the hierarchy of the constructible universe $L_\alpha\subseteq L_\beta$ are covering end-extensions; every transitive set end-extends to all its supersets. The set-theoretic universe $V$ is an end-extension of the constructible universe $L$ and every forcing extension $M[G]$ is an end-extension of its ground model $M$, though neither of these extensions are covering. In particular, one should not confuse end-extensions with rank-extensions, where all the new sets come in only at higher ranks. Observe, however, that every $\Sigma_2$-elementary end-extension is a rank-extension and that every rank-extension modeling \ZF\ is a covering end-extension.

In the proof of the main theorem, we shall make key use of the fact that every countable model of set theory has an elementary end-extension.

\begin{theorem}[Keisler--Morley~\cite{KeislerMorley1968:ElementaryExtensionsOfModelsOfSetTheory}]
Every countable model of \ZF\ has an elementary end-extension.
\end{theorem}

For models of \ZFC, this can be proved by means of a suitable definable ultrapower, in the language of set theory augmented with a generic well-order. The general \ZF\ case can be proved using the omitting types theorem in order to ensure the end-extension property. The result is known not to generalize to uncountable models of set theory. For example, if $\kappa$ is the least inaccessible cardinal, then $\<V_\kappa,\in>$ has no elementary end-extension. 

To prove the main theorem we will make use of some reflection and absoluteness facts, which we state here precisely.
The first is the reflection principle: for each formula $\varphi(x)$, allowing parameters, and each set $a$ there is a transitive set $m \ni a$ so that $\varphi(x)$ reflects from $V$ to $m$. We will also use the \Levy\ reflection theorem, asserting that $\Sigma_1$-assertions in set theory are absolute to $L_{\omega_1^L}$, and the Shoenfield absoluteness theorem, asserting that $\Sigma^1_2$-assertions about reals are absolute between models with the same countable ordinals.

Finally, we will need the \Godel--Carnap fixed-point lemma for our definition of the $\Sigma_1$-definable universal finite sequence. We give here a formulation of this lemma appropriate for our context.

\begin{lemma}[\Godel--Carnap fixed-point lemma]
Let $\varphi(x,\bar y)$ be a formula in the language of set theory. Then there is a formula $\psi(\bar y)$ so that $\KP$ proves $\forall \bar y\ \psi(\bar y) \iff \varphi(\gcode{\psi},\bar y)$.
\end{lemma}

\section{A new proof of the Barwise extension theorem}\label{Section.A-new-proof-of-the-Barwise-extension-theorem}

Our proof of the main theorem grows out of ideas in a certain new proof of the classical Barwise extension theorem, which we shall provide in this section as a warm-up to the main theorem.

\begin{theorem}[Barwise extension theorem~\cite{Barwise1971:InfinitaryMethodsInTheModelTheoryOfSetTheoryLC69}] \label{Theorem.Barwise-extension}
 Every countable model of \ZF\ has an end-extension to a model of $\ZFC+V=L$.
\end{theorem}

The Barwise extension theorem is simultaneously (i) a technical culmination of the pioneering methods of Barwise in admissible set theory and infinitary logic, including the Barwise compactness and completeness theorems and the admissible cover, but also (ii) one of those rare mathematical theorems that is saturated with significance for the philosophy of mathematics and particularly the philosophy of set theory. The first author discussed these philosophical aspects at length in~\cite{Hamkins2014:MultiverseOnVeqL}.

The new development here regarding this theorem is a new direct proof due to Hamkins, with the interesting feature---downright surprising in light of (i) above---that it makes no use of Barwise compactness and indeed no use of infinitary logic at all. Instead, the proof uses only classical methods of descriptive set theory, namely, the representation of $\Pi^1_1$-definable sets with well-founded trees, the \Levy{} and Shoenfield absoluteness theorems, the reflection theorem, and the Keisler--Morley theorem on elementary end-extensions of countable models of set theory.

\begin{proof}[Proof (Hamkins~\cite{Hamkins.blog2018:A-new-proof-of-the-barwise-extension-theorem})]
Suppose that $M$ is a countable model of $\ZF$ set theory. Consider first the easier case, where $M$ is $\omega$-nonstandard. For any standard natural number $k$, the reflection theorem ensures that there are arbitrarily high $L_\alpha^M$ satisfying the finite theory fragment $\ZFC_k$, and so every countable transitive set $m\in L^M$ has an end-extension to a model of $\ZFC_k+V=L$. By overspill, there must be some nonstandard $k$ for which $L^M$ thinks that every countable transitive set $m$ has an end-extension to a countable model of $\ZFC_k+V=L$. This is a $\Pi^1_2$-statement about $k$, which will therefore also be true in $M$, by the Shoenfield absoluteness theorem. By the Keisler--Morley theorem, $M$ has an elementary end-extension $M^+$. Let $\theta$ be a new ordinal on top of $M$, and let $M^+[G]$ be a forcing extension in which $m=V_\theta^{M^+}$ becomes countable. Since the $\Pi^1_2$-statement transfers to $M^+$ and then to $M^+[G]$, there is in $M^+[G]$ an end-extension of $\langle m,\in^{M^+}\rangle$ to a model $\langle N,\in^N\rangle$ that $M^+[G]$ thinks satisfies $\ZFC_k+V=L$. Since $k$ is nonstandard, this theory includes all of the actual $\ZFC$ axioms, and since $m$ end-extends $M$, we have found an end-extension of $M$ to a model of $\ZFC+V=L$, as desired.

Consider now the harder case, where $M$ is $\omega$-standard. Let $M^+$ be an elementary end-extension of $M$, and consider $m=V_\theta^{M^+}$, where $\theta$ is a new ordinal above $M$. If $\langle m,\in^{M^+}\rangle$ has an end-extension to a model of $\ZFC+V=L$, then we're done, since such a model would also end-extend $M$. So assume toward contradiction that there is no such end-extension of $m$. Let $M^+[G]$ be a forcing extension in which $m$ has become countable. The assertion that $m$ has no end-extension to a model of $\ZFC+V=L$ is actually true and hence true in $M^+[G]$. This is a $\Pi^1_1$-assertion there about the real coding $m$. Every such assertion has a canonically associated tree, which is well-founded exactly when the statement is true. Since the statement is true in $M^+[G]$, this tree has some countable rank $\lambda$ there. Since these models have the standard $\omega$, the tree associated with the statement is the same for us as inside the model, and since the statement is actually true, the tree is actually well-founded. So the rank $\lambda$ comes from the well-founded part of the model.

If $\lambda$ happens to be countable in $L^{M^+}$, then consider the assertion, ``there is a countable transitive set, such that the assertion that it has no end-extension to a model of $\ZFC+V=L$ has rank $\lambda$.'' This is a $\Sigma_1$-assertion, since it is witnessed by the countable transitive set and the ranking function of the tree associated with the non-extension assertion. Since the parameters are countable in $L^{M^+}$, it follows by \Levy\ reflection that the statement is true in $L^{M^+}$. So $L^{M^+}$ has a countable transitive set, such that the assertion that it has no end-extension to a model of $\ZFC+V=L$ has rank $\lambda$. But since $\lambda$ is actually well-founded, the statement would have to be actually true; but it isn't, since $L^{M^+}$ itself is such an extension, a contradiction.

So we may assume $\lambda$ is uncountable in $L^{M^+}$. In this case, since $\lambda$ was actually well-ordered, it follows from the fact that $M \prec M^+$ that $L^M$ is well-founded beyond its $\omega_1$. Consider the statement ``there is a countable transitive set having no end-extension to a model of $\ZFC+V=L$.'' This is a $\Sigma^1_2$-sentence, which is true in $M^+[G]$ by our assumption about $m$, and so by Shoenfield absoluteness, it is true in $L^{M^+}$ and, because $M \prec M^+$, hence also true in $L^M$. So $L^M$ thinks there is a countable transitive set $b$ having no end-extension to a model of $\ZFC+V=L$. This is a $\Pi^1_1$-assertion about $b$, whose truth is witnessed in $L^M$ by a ranking of the associated tree. Since this rank would be countable in $L^M$ and this model is well-founded up to its $\omega_1$, the tree must be actually well-founded. But this is impossible, since it is not actually true that $b$ has no such end-extension, since $L^M$ itself is such an end-extension of $b$. Contradiction.
\end{proof}

Observe that if one desires to have a covering end-extension this is easily obtained by using the Keisler--Morley theorem to extend the end-extension further to an elementary end-extension.

We should like to note conversely that the Barwise extension theorem is an immediate consequence of our main theorem, which can be seen as a generalization of it. Specifically, the Barwise theorem follows from statement (4) of the main theorem, simply by taking $\ZFbar$ to be the theory $\ZFC+V=L$. Every model of \ZF\ has an inner model of $\ZFC+V=L$, and so the main theorem shows that not only does every countable model of \ZF\ have an end-extension to a model of $\ZFC+V=L$, but it has such an extension in which the original model is covered, making every element of it countable, and for which we may also control the universal sequence however we desire.

Indeed, the main theorem provides an array of strengthenings of the Barwise extension theorem. For example, Corollary~\ref{Corollary.Applications-of-main-theorem} shows that every countable model of set theory with a measurable cardinal has a covering end-extension to a model of $\ZFC+V=L[\mu]$; every countable model of set theory with extender-based large cardinals has a covering end-extension to a model satisfying $V=L[\vec E]$; and every countable model of set theory with infinitely many Woodin cardinals and a measurable above has a covering end-extension to a model of $\ZF+\text{AD}+{V=L(\mathbb{R})}$. And there are many further interesting cases, where one end-extends a given countable model of set theory to a model satisfying a theory holding in a sufficiently strong submodel of the original model.
\goodbreak

\section{The \texorpdfstring{$\Sigma_1$}{Sigma\textunderscore{}1}-definable universal finite sequence for end-extensions}

Let us now prove the main theorem. We restate it here for the benefit of the reader.

\begin{main}
For any computably axiomatizable theory $\ZFbar$ extending \ZF, there is a $\Sigma_1$-definable finite sequence
$$a_0,\ a_1,\ \ldots,\ a_n$$
with the following properties:
\begin{enumerate}
\item $\ZF$ proves that the sequence is finite.
\item In any transitive model $M$ of $\ZFbar$, the sequence is empty.
\item If $M$ is a countable model of $\ZFbar$ in which the sequence is $s$ and $t \in M$ is any finite sequence (in the sense of $M$) extending $s$, then there is a covering end-extension of $M$ to a model $N \models \ZFbar$ in which the sequence is exactly $t$, and in which every set of $M$ is countable.
\item Indeed, for statements (2) and (3), it suffices merely that $M\satisfies\ZF$ and end-extends a submodel $W\satisfies\ZFbar$ of height at least $(\omega_1^L)^M$.
\end{enumerate}
\end{main}

\begin{proof}
We undertake an analogue of the universal finite set construction of~\cite{HamkinsWoodin:The-universal-finite-set}, a construction which is itself a set-theoretic generalization of the universal algorithm~\cite{Woodin2011:A-potential-subtlety-concerning-the-distinction-between-determinism-and-nondeterminism, Blanck2017:Dissertation:Contributions-to-the-metamathematics-of-arithmetic, BlanckEnayat2017:Marginalia-on-a-theorem-of-Woodin, Hamkins:The-modal-logic-of-arithmetic-potentialism}, but combining it with the ideas of Hamkins's proof of the Barwise extension theorem in section~\ref{Section.A-new-proof-of-the-Barwise-extension-theorem}.

Fix a computable enumeration of the theory $\ZFbar$, to be used in all the models of set theory in which we refer to this theory, and let $\ZFbar_k$ be the finite theory consisting of the first $k$-many axioms appearing in this enumeration.

As explained in~\cite{Hamkins:The-modal-logic-of-arithmetic-potentialism}, it will suffice for us to establish a weaker extension property, the \emph{adding-one} extension property, which asserts that if $s$ is the sequence defined in $M$ then for any object $a$ in $M$ there is a covering end-extension $N$ in which the defined sequence is $s^\smallfrown a$. That is, the adding-one extension property allows one to extend the sequence by appending exactly one new object on the end, rather than a finite sequence. By iterating this property, of course, one can eventually append any standard-finite sequence, but such an inductive argument, it turns out, cannot establish the full extension property for nonstandard-finite extensions, for there are sequences with the adding-one extension property that do not have the full extension property.\footnote{For example, the main process $A$ sequence we define in this proof has the adding-one extension property, but not the arbitrary extension property, because the total length of the sequence will be bounded by $k_0$, a bound which is known once the first stage is successful.} Nevertheless, from any sequence with the adding-one extension property, one can define a new sequence with the arbitrary extension property. For example, one could simply concatenate all the finite sequences that appear on the given adding-one sequence. In this way, by adding one object to the original sequence, we can add an arbitrary (possibly nonstandard) finite sequence to the derived concatenated sequence. Therefore, for the rest of this proof, we shall aim only for the adding-one extension property.

As in~\cite{HamkinsWoodin:The-universal-finite-set}, we shall describe two set-theoretic processes, $A$ and $B$. These processes are intended to be run inside, respectively, $\omega$-nonstandard models and $\omega$-standard models and will then be merged in a way that provides a single definition fulfilling the desired properties.

\textbf{Process $A$.} We start by describing process $A$, intended to be run inside $\omega$-nonstandard models as an internal process, using whatever nonstandard  natural numbers are to be found there. The process proceeds in a sequence of stages, placing one object onto the sequence at each successful stage. Stage $n$ succeeds and $a_n$ is defined, if there is a transitive set $m_n$, countable in $L$ and containing as elements all earlier $m_i$ for stages $i<n$, and a natural number $k_n$, smaller than all earlier $k_i$, such that $a_n \in m_n$ and the structure $\<m_n,\in>$ has no covering end-extension to a model $\<N,\in^N>$ satisfying $\ZFbar_{k_n}$, making every set in $m_n$ countable, and placing this very object $a_n$ onto its own $A$ sequence at stage $n$ as the last element. In slogan form: we place an object onto the sequence, if we find a countable transitive set having no covering end-extension in which we would have done so as the next and last element. This is altogether a $\Delta_1$-property of the data $(m_n,a_n,k_n)$, since the non-existence of such a model $N$ is a $\Pi^1_1$-assertion, whose truth can be verified by the ordinal ranking of the canonically associated well-founded tree. In particular, if there is such an end-extension, then there is one in $L$ and countable in $L$. The process officially accepts and uses the triple $(m_n,a_n,k_n)$, whose witnesses for this property are $L$-least. In other words, we accept the first triple to be verified in the $L$ hierarchy, and use this triple to define $a_n$. The map $n\mapsto (m_n,k_n,a_n)$ is accordingly $\Sigma_1$-definable.

Although the definition may at first appear circular---we define $a_n$, after all, by reference to the definition of the process $A$ sequence inside $N$---one may use the \Godel--Carnap fixed point lemma to find a definition $\psi(n,a)$ for the map $n\mapsto a_n$ that solves this recursion, allowing $\psi$ as described to refer to its own \Godel\ code $\gcode{\psi}$. The same method is used in~\cite{HamkinsWoodin:The-universal-finite-set}, and one should view this as analogous to the common use of the Kleene recursion theorem in computability-theoretic arguments, including its use in the universal algorithm~\cite{Woodin2011:A-potential-subtlety-concerning-the-distinction-between-determinism-and-nondeterminism, Hamkins:The-modal-logic-of-arithmetic-potentialism}.

Since the natural numbers $k_n$ are descending, there will be only finitely many successful stages, and so the sequence will be finite.

We claim that in the relevant models, every $k_n$ arising from a successful stage is nonstandard; in particular, $\omega$-standard models will have no successful process $A$ stages. To see this, consider any countable model of set theory $M\satisfies\ZF$, which end-extends a submodel $W\satisfies\ZFbar$ of height at least $(\omega_1^L)^M$, and let $n$ be the last successful stage in $M$. Since the definition is $\Sigma_1$, By \Levy{} reflection it is absolute between $M$ and $W$. For any standard $k$, by reflection the theory $\ZFbar_k$ is true in a transitive set $N\in W$, large enough to include $m_n$ as a countable set and the witnesses for the successful stages. Consequently, $\<N,\in^W>\satisfies \ZFbar_k$ and thinks object $a_n$ was added at stage $n$ as the last object on the sequence. If $k_n \leq k$, this would violate the success of stage $n$ in $M$, since $m_n$ would have had such a covering end-extension after all. Therefore $k<k_n$ for every standard $k$, and so $k_n$ is nonstandard.

It remains to verify the extension property. Let $M$ be a countable model of \ZF\ set theory in which the sequence is $s$. (We don't actually need the submodel $W$ of $\ZFbar$ for this part of the argument, provided that the earlier $k_i$, if any, are nonstandard.) Let $n$ be the first unsuccessful stage. Let $k$ be nonstandard, but smaller than all previous $k_i$. Since stage $n$ did not succeed, it follows that for any transitive set $m$ countable in $L^M$ and containing the earlier $m_i$, and for any set $a\in m$, the structure $\<m,\in>^M$ does have a covering end-extension in $M$ to a model making every set in $m$ countable and satisfying $\ZFbar_k +$``object $a$ was placed onto the sequence at stage $n$, the last successful stage,'' since otherwise stage $n$ would have succeeded. Further, for a given $m$, the existence of such an end-extension is a $\Sigma_1$-property of the data $(m,a,k)$, and so we may find the extensions in $L^M$. Therefore, $L^M$ thinks that for every countable set $a$ and every countable transitive set $m$, there is a covering end-extension of $\<m,\in>$ to a model $\<N,\in^N>$ making every set in $m$ countable and satisfying $\ZFbar_k+$``object $a$ is placed onto the sequence at stage $n$, the last successful stage.'' This is a $\Pi^1_2$-assertion ($\forall a,m\,\exists N\ldots$), which by Shoenfield is therefore also true in $M$ itself. By the Keisler--Morley theorem, $M$ has an elementary end-extension $M^+$, and so the statement will also be true in $M^+$, as well as in all its forcing extensions, by Shoenfield absoluteness. Let $\theta$ be an ordinal of $M^+$ above $M$, and let $M^+[G]$ be a forcing extension making $V_\theta^{M^+}$ countable. Thus, for any object $a$ in $M$, we have a countable transitive set $m=V_\theta^{M^+}$ in $M^+[G]$, which by our observations must therefore have a covering end-extension $N$ in $M^+[G]$ making every set in $m$ countable and in which $\ZFbar_k$ holds, plus the assertion that process $A$ places object $a$ onto the sequence at stage $n$, the last successful stage. Since $N$ end-extends $V_\theta^{M^+}$, which end-extends $M$, and since $k$ is nonstandard, we have therefore found the desired covering end-extension of $M$ to a model of $\ZFbar$ making every set in $M$ countable and placing $a$ as the next and last element on the sequence.

\textbf{Process $B$.} We turn now to process $B$, intended to be run as an internal process inside $\omega$-standard models, using whatever (possibly nonstandard) ordinals are to be found there. The process will proceed in a sequence of stages, just as before, with each successful stage adding one object $a_n$ to the sequence. Stage $n$ is successful and $a_n$ is defined, if there is a transitive set $m_n$ countable in $L$ and a countable ordinal $\lambda_n$, with $m_i \in m_n$ and $\lambda_n < \lambda_i$ for all earlier stages $i<n$, and with $\lambda_n \in m_n$ and a set $a_n\in m_n$, such that the structure $\<m_n,\in>$ has no covering end-extension to a model $\<N,\in^N>$ making every set in $m_n$ countable and satisfying $\ZFbar+$``process $B$ places object $a_n$ on the sequence at stage $n$, the last successful stage,'' and furthermore, this property about $(m_n,a_n,\lambda_n)$, which has complexity $\Pi^1_1$, has rank $\lambda_n$ in the canonical representation of $\Pi^1_1$-assertions by well-founded trees. The acceptability of the data triple is a $\Sigma_1$-property, witnessed by the ranking function, and we officially accept the triple whose witness appears least in the $L$ order. So the map $n\mapsto a_n$ is $\Sigma_1$-definable.

Once again, the circularity in the definition can be removed by the \Godel--Carnap fixed-point lemma. Since the ordinals $\lambda_n$ are descending, there will be only finitely many successful stages, and so the sequence is finite.

We claim that the ordinals $\lambda_n$ in the relevant models are nonstandard. Suppose that $M$ is a countable model of $\ZF$ that end-extends a submodel $W$ of $\ZFbar$ of height at least $(\omega_1^L)^M$, and let $n$ be the last successful stage in $M$. Since the definition is $\Sigma_1$, it is absolute between $M$ and $W$, and so $W$ satisfies $\ZFbar$ plus the assertion that the object $a_n$ is placed onto the sequence at stage $n$, the last successful stage. So the statement asserting at stage $n$ that $m_n$ has no covering end-extension and so on is actually false, since $W$ itself is such an extension. So the tree representing the statement, which is determined the same in $M$ as for us in the set-theoretic background because $M$ is an $\omega$-model, is not actually well-founded. But $M$ thought it was well-founded with rank $\lambda_n$, and so $\lambda_n$ must be nonstandard.

Let us now prove the extension property for process $B$ in countable $\omega$-standard models $M\satisfies\ZF$ end-extending a submodel $W\satisfies \ZFbar$ of height at least $(\omega_1^L)^M$. Suppose the sequence is $s$ in $M$. Let $n$ be the first unsuccessful stage, and consider any set $a$ in $M$. By the Keisler--Morley theorem, we may find a countable elementary end-extension of $M$ to a model $M^+$. Let $\theta$ be an ordinal of $M^+$ above $M$, and let $M^+[G]$ be a forcing extension in which $m=V_\theta^{M^+}$ is made countable. If $\langle m,\in^{M^+}\rangle$ has a covering end-extension to a model of $\ZFbar$ in which every set of $m$ is made countable and object $a$ is placed onto the sequence at stage $n$ as the last successful stage, then we'd be done, since this would also serve as the desired end-extension of the original model $M$. So let us assume toward contradiction that there is no such end-extension of $\langle m,\in^{M^+}\rangle$. This is a $\Pi^1_1$-assertion about $m$ in $M^+[G]$, which therefore has some ordinal rank $\lambda$ there in the representation of $\Pi^1_1$-assertions by well-founded trees. Since the statement is really true (in the set-theoretic background of our universe), it follows that $\lambda$ must be in the well-founded part of $M^+[G]$. In particular, $\lambda<\lambda_i$ all $i<n$, since those ordinals are nonstandard. So $M^+[G]$ thinks that ``there is a countable transitive set $m$ containing all the earlier $m_i$ and the witnesses for the success of the $\Sigma_1$ definition at those earlier stages and an element $a\in m$, such that the assertion that $\<m,\in>$ has no covering end-extension to a model making every set in $m$ countable and satisfying $\ZFbar$ $+$ ``object $a$ is placed onto the sequence at stage $n$, the last successful stage'' has ordinal rank $\lambda$.'' This is a $\Sigma_1$-assertion about $\lambda$ and the countable set containing the data and witnesses for the earlier successful stages of the process, since it is witnessed by the countable transitive set $m$ and $a\in m$ and the ranking function of the tree that is canonically associated with the non-extension assertion about $m$.

If $\lambda$ happens to be countable in $L^{M^+}$, then all the parameters of the assertion are countable in $L^M$, and so it follows by \Levy\ reflection that the statement is true in $L^{M^+}$ and hence in $L^M$. So $L^M$ has a countable transitive set $m$, containing the witnesses for the earlier successful stages and an object $a\in m$, such that the assertion that $m$ has no covering end-extension to a model of $\ZFbar$ making every set in $m$ countable and in which $a$ is placed at stage $n$, the last stage, has rank $\lambda$. But in this case, stage $n$ would have succeeded, contradicting our assumption that it was the first unsuccessful stage.

So we may assume $\lambda$ is uncountable in $L^{M^+}$. In this case, since $\lambda$ was actually well-ordered, it follows that $L^M$ is well-founded beyond its $\omega_1$. This implies that there could have been no earlier successful stages in $M$, since the $\lambda_i$ must be nonstandard countable ordinals in $L^M$. Thus, $n=0$, or in other words, there is nothing yet on the sequence. Consider the statement ``there is a countable transitive set $m$ with an element $a\in m$, such that $m$ has no covering end-extension to a model of $\ZFbar$ making every set in $m$ countable, in which the sequence is $\<a>$, having exactly one successful stage.'' This is a $\Sigma^1_2$ sentence, which is true in $M^+[G]$ by our assumption about $m$, and so by Shoenfield absoluteness, it is true in $L^{M^+}$ and hence also in $L^M$. So $L^M$ thinks there is a countable transitive set $m$ with an element $a\in m$, such that $m$ has no covering end-extension to a model of $\ZFbar$ in which every set in $m$ is countable and the sequence is exactly $\<a>$. But in this case, there would have been a successful stage after all, contrary to what we proved earlier.

Altogether, therefore, we conclude that $V_\theta^{M^+}$ must have had the desired end-extension after all, and so we've verified the extension property of process $B$ for the relevant countable $\omega$-standard models of set theory.

{\bf Process C.} We shall now merge processes $A$ and $B$ into a single process $C$, providing a single uniform $\Sigma_1$-definition that will work in all the relevant countable models of set theory. Process $C$ proceeds in stages. At each successful stage, it will place one new object onto the sequence, either for an $A$-reason or for a $B$-reason. The $A$-reasons will involve data $(m_n,a_n,k_n)$, and the $B$-reasons will involve data $(m_n,a_n,\lambda_n)$, where we insist that $m_n$ is a transitive set countable in $L$ and $m_i \in m_n$ for the $m_i$ used for either reason at earlier stages $i<n$, that $a_n\in m_n$, and that $k_n<k_i$ for the earlier $A$-reason stages $i<n$, if this is an $A$-reason stage, and $\lambda_n<\lambda_i$ for the earlier $B$-reason stages, if this is a $B$-reason stage; once there has been a successful $A$-reason stage, then we proceed only with the $A$-reason stages, but the $m_i$ for the $A$-reasons must see all the data for the previously successful $B$-reason stages. In particular, each $B$-reason $\lambda_j$ should be an element of any later $A$-reason $m_i$. As before, the data is acceptable if $m_n$ has no covering end-extension to a model $\<N,\in^N>$ making every set in $m_n$ countable and satisfying the relevant theory (using $\ZFbar_{k_n}$ at $A$-reason stages and full $\ZFbar$ at $B$-reason stages) plus the assertion that this new process $C$ has exactly one new successful stage, placing object $a_n$ onto the process $C$ sequence, and furthermore, in the $B$-reason case, the assertion that there is no such $N$ has rank $\lambda_n$. The acceptability of such a data triple is a $\Sigma_1$-property, since it is verified for each reason type by the rankings of certain trees, as with processes $A$ and $B$ above. A stage is successful if any data is verified as acceptable for it, and we use the data and the reason type whose witness appears first in the $L$-order. This defines the process $C$ sequence $n\mapsto a_n$, which has complexity $\Sigma_1$.

We complete the argument by observing that our process $A$ and $B$ analysis goes through for process $C$. The sequence is finite since the $k_n$ and $\lambda_n$ can go down only finitely many times. At any $A$-reason stage, the number $k_n$ must be nonstandard, and at any $B$-reason stage, the ordinal $\lambda_n$ must be nonstandard. In any $\omega$-nonstandard model, we achieve the extension property for process $C$ for the same reasons as process $A$; and in any $\omega$-standard model, we achieve the extension property for process $C$ just as with process $B$.
\end{proof}

As an immediate corollary, let us see that we can get a universal finite sequence for extensions in the $L$-hierarchy, which was our original motivation for this project. If $M$ and $N$ are models of $V = L$, say that $N$ is an \emph{$L$-extension} of $M$ if $M$ is an initial segment in the $L_\alpha$-hierarchy for $N$. Compare this to the notion of a rank-extension, where $N$ is a rank-extension of $M$ if $M$ is an initial segment in the $V_\alpha$-hierarchy for $N$.

\begin{corollary}[Universal finite sequence for $L$-extensions]
There is a $\Sigma_1$-definable sequence
$$a_0, \ldots, a_n$$
with the following properties:
\begin{enumerate}
\item $\ZF$ proves the sequence is finite.
\item In any transitive model of $\ZF$, the sequence is empty.
\item If $M$ is a countable model of $\ZFC+V=L$ in which the sequence is $s$ and $t \in M$ is any finite sequence extending $s$, then there is an $L$-extension $N \satisfies \ZFC+V=L$ of $M$ in which every set in $M$ is countable and the sequence is exactly $t$.
\end{enumerate}
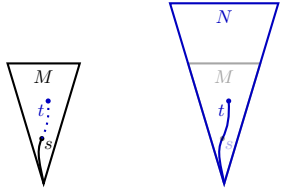
$$\begin{tikzpicture}[xscale=.6,scale=.4]
 \draw[thick] (0,0) -- (-2,4) -- node[below,scale=.7] {$M$} (2,4) -- cycle;
 \draw[thick] (0,0) to[out=105,in=-130] (-.1,1.5) node[circle,fill=black,scale=.2,label={[label distance=-5pt,scale=.7]below right:$s$}] {};
 \draw[thick,dotted,blue!75!black] (-.1,1.5) to[out=50,in=-85] (.25,2.75) node[circle,fill=blue!75!black,scale=.2,label={[label distance=-4pt,scale=.7]below left:$t$}] {};
\begin{scope}[shift={(10,0)}]
 \draw[thick,opacity=.35] (0,0) -- (-2,4) -- node[below,scale=.7] {$M$} (2,4) -- cycle;
 \draw[thick,blue!75!black] (0,0) -- (-3,6) -- node[below,scale=.7] {$N$} (3,6) -- cycle;
 \draw[thick,blue!75!black] (0,0) to[out=105,in=-130] (-.1,1.5) node[circle,fill opacity=.35, fill=black,scale=.2,label={[opacity=.35,label distance=-5pt,scale=.7]below right:$s$}] {} to[out=50,in=-85] (.25,2.75) node[circle,fill=blue!75!black,scale=.2,label={[label distance=-4pt,scale=.7]below left:$t$}] {};
\end{scope}
\end{tikzpicture}$$
\end{corollary}

\begin{proof}
This follows from the main theorem using $\ZFbar = \ZFC+V=L$ and the observation that, by the absoluteness of the $L_\alpha$-hierarchy, among models of $V=L$ being an end-extension is equivalent to being an $L$-extension.
\end{proof}

In the context of models of arithmetic, the universal algorithm exhibited the extension property with respect to all models of arithmetic, not just the countable models. But in the context of set theory, we achieve the extension property of the universal finite sequence in the main theorem only for the countable models of set theory. This was because at certain points in the proof, we made key use of some results, such as the Keisler--Morley theorem, which hold for the countable models of set theory but not generally.

It is therefore natural to inquire whether this limitation can be removed. Can there be a $\Sigma_1$-definable universal finite sequence exhibiting the extension property for all models of set theory, including uncountable models? That would make the situation analogous with the universal algorithm for models of arithmetic.

Alas, the answer is no. If there are sufficiently well-founded uncountable models, then there can be no generalization of the universal finite sequence to uncountable models of set theory. Specifically, suppose that $M$ is model of set theory that is well-founded at least to true $\omega_1$. We claim that no end-extension of $M$ can have new $\Sigma_1$-facts, and in particular, the universal sequence cannot gain new elements in an end-extension. The reason is that if $N$ is an end-extension of $M$, then any $\Sigma_1$-fact true in $N$ is true in $L_{{\omega_1}^L}^N$ by \Levy\ absoluteness, and this latter model cannot exceed the true $L_{{\omega_1}^L}$, for then we would have collapsed the true $\omega_1$ to be countable. 

Meanwhile, the question would remain whether that is the only kind of exception.

\begin{question}
Does the main theorem generalize to the case of uncountable models of set theory, whose well-founded part does not include the true $\omega_1$?
\end{question}

\section{A refined version of the main theorem, with applications}

We should like now to prove a refined version of the main theorem, by isolating exactly the features one needs of the theory in order to undertake the proof of the main theorem.

One feature of the theory $\ZFbar$ that we used was that it satisfied a form of the reflection principle. It will turn out, however, that a weak form of reflection will suffice. Specifically, let us define that a theory $T$ is \emph{reflexive to transitive sets}, if for any finitely many axioms of $T$, the theory proves that those axioms hold cofinally in the transitive sets. The difference between this and full reflection is that here, we reflect only assertions of the theory, rather than arbitrary true assertions.

The theory \ZF\ and all its extensions are reflexive to transitive sets, of course, by the reflection principle. But also, the theory $\ZFCm+V=L$ is reflexive to transitive sets, where $\ZFCm$ is Zermelo--Fraenkel set theory without the power set axiom,\footnote{By $\ZFCm$, we mean here, as one should, the theory axiomatized with the separation and collection schemata rather than the replacement schema, and with the well-ordering theorem rather than the axiom of choice. These are not equivalent without the power set axiom, as proved by Zarach~\cite{Zarach1996:ReplacmentDoesNotImplyCollection,Zarach1982:Unions-of-ZFm-models};
see also~\cite{GitmanHamkinsJohnstone2016:WhatIsTheTheoryZFC-Powerset?}.} 
since one can reflect along the $L$-hierarchy. Meanwhile, \KP\ is not reflexive, nor is any finitely axiomatizable theory. In general, the reflexive property is weaker than the reflection theorem for $T$, since models of $\ZFCm$, for example, do not necessarily satisfy the reflection theorem, as proved in~\cite{FriedmanGitmanKanovei2018:A-model-of-AC-not-DC}, but nevertheless, the theory is reflexive to transitive sets, since if $M\satisfies\ZFCm$ and $a\in M$, then we can code $a$ with a set of ordinals $A\in M$ and consider $L[A]^M$, which is a model of $\ZFCm$ containing $a$ and having an $\Ord$-hierarchy, which is enough to prove the reflection theorem. So we get transitive models of the form $L_\alpha[A]$ satisfying any desired finite fragment of $\ZFCm$ and containing the set $a$.

Another consequence of \ZF\ that we used in the main argument was the characterization of $\Pi^1_1$-assertions via well-founded trees. According to~\cite[Lemma~V.1.8]{Simpson2009:SubsystemsOfSecondOrderArithmetic}, however, this characterization is provable in $\ACA_0$. In particular, it will be provable in any set theory extending \KP, and \KP\ is sufficient to provide ordinal rankings of those trees.

It follows that $\Pi^1_1$-assertions about reals are absolute between a model $M$ of $\KP$ and its end-extensions $N\satisfies\KP$, since if the statement is true in the smaller model $M$, then it has the ordinal ranking of the tree, which will still exist in the larger model, verifying the $\Pi^1_1$-statement there; and if the statement fails in the smaller model, then it has the counterexample real, which will still exist in the larger model.

The \Levy\ reflection theorem, asserting that $\Sigma_1$-assertions in set theory are absolute to $L_{\omega_1^L}$, follows from $\Pi^1_1$-absoluteness. The upward direction is immediate; conversely, if a $\Sigma_1$-assertion is true in $M$, then by \Lowenheim--Skolem there is a countable witness, and this amounts to the truth of a $\Sigma^1_1$-assertion, which is therefore absolute to $L_{{\omega_1}^L}^M$.

The Shoenfield absoluteness principle, however, the assertion that $\Sigma^1_2$-assertions about reals are absolute between models with the same countable ordinals, appears to be strictly stronger. Shoenfield absoluteness is provable in $\Pi^1_1\text{-CA}_0$, according to~\cite[theorem~VII.4.14]{Simpson2009:SubsystemsOfSecondOrderArithmetic}. Because $\Pi^1_1\text{-CA}_0$ is stronger than \KP, we shall need to impose an extra assumption on the theory $S$ concerning this.

In the proof of the main theorem, we had used the Keisler--Morley theorem, stating that every countable model $M$ of \ZF\ has an elementary end-extension $M^+$. A closer inspection of the proof, however, will reveal that a somewhat weaker property will suffice. Specifically, what we shall need in Theorem~\ref{Theorem.Universal-sequence-for-suitable-theories} is merely that for every countable model $M$ of the theory, there is a $\Sigma_1$-elementary covering end-extension $M^+$ satisfying the theory (or at least satisfying the Shoenfield absoluteness property).

While all extensions of \ZF\ have this property, it turns out that \KP\ does not. Specifically, $L_{\omega_1^\mathrm{CK}}$, the minimum transitive model of \KP\ plus the axiom of infinity, can have no such covering $\Sigma_1$-elementary end-extension (Ali Enayat proved this in email correspondence). Meanwhile, Kaufmann~\cite{Kaufmann1981.On-existence-of-sigma_n-end-extensions} shows that satisfying $\Sigma_n$-collection implies the existence of a $\Sigma_n$-elementary end-extension, for $n\geq 2$. An easy modification of his argument gives that we may assume the end-extension to be covering.

With the above discussion in mind, we make the following definition.

\begin{definition}\label{Definition.Suitable-theories}\rm
A pair of theories $S$ and $T$ are \emph{suitable} for the $\Sigma_1$-definable universal sequence, provided that theory $S$ has the following properties
\begin{enumerate}
  \item $S$ extends \KP.
  \item Countable models of $S$ satisfy $\Sigma^1_2$-absoluteness to $L$ and to their forcing extensions.
  \item Every countable model $M\satisfies S$ has a $\Sigma_1$-elementary covering end-extension $M^+$, which also satisfies the $\Sigma^1_2$-absoluteness to its $L$ and forcing extensions.
\end{enumerate}
and the theory $T$ has the following properties
\begin{enumerate}
  \item $T$ extends \KP.
  \item $T$ is computably enumerable.
  \item $T$ is reflexive to transitive sets.
\end{enumerate}
A theory $T$ is \emph{suitable} all by itself, if it has all the properties, making a suitable pair with itself.
\end{definition}

It is clear that \ZF, and indeed any computably enumerable extension of \ZF, is suitable. Our earlier discussion shows that $\ZFm$ also is suitable.

We can now state a version of the main theorem for these suitable theories.\goodbreak

\begin{theorem}\label{Theorem.Universal-sequence-for-suitable-theories}
Assume that theories $S$ and $T$ are suitable for the $\Sigma_1$-definable universal finite sequence, as defined above. Then there is a $\Sigma_1$-definable finite sequence
$$a_0,\ a_1,\ \ldots,\ a_n$$
with the following properties:
\begin{enumerate}
\item $S$ proves that the sequence is finite.
\item In any transitive model $M$ of $T$, the sequence is empty.
\item If $M$ is a countable model of $T$ in which the sequence is $s$ and $t \in M$ is any finite sequence (in the sense of $M$) extending $s$, then there is a covering end-extension of $M$ to a model $N \models T$ in which the sequence is exactly $t$, and in which every set in $M$ becomes countable in $N$.
\item Indeed, for statements (2) and (3), it suffices merely that $M\satisfies S$ and end-extends a submodel $W\satisfies T$ of height at least $(\omega_1^L)^M$. \qed
\end{enumerate}
\end{theorem}

This is proved just as the main theorem, as our modifications in Definition~\ref{Definition.Suitable-theories} exactly capture what was needed in the proof.

\begin{question}
Are there more generous notions of suitability, for which the conclusions of the main theorem can still be established?
\end{question}

We can drop the reflexive-to-transitive-sets requirement on the theory $T$, for example, if we are willing to have the universal extension property for the sequence only in $\omega$-standard models, because our process $B$ analysis did not require that reflexive property. But can we drop that requirement in any case?

Just as the main theorem had the Barwise extension theorem as an immediate consequence, we also get the following consequence of the generalized version of the theorem.

\begin{corollary} \label{Corollary.Suitable-resurrection}
 Assume $S$ and $T$ are suitable theories, in the sense of definition~\ref{Definition.Suitable-theories}. Then for every countable model $M\satisfies S$ end-extending a submodel $W\satisfies T$ of height at least $(\omega_1^L)^M$, there is a covering end-extension of $M$ to a model of $T$, and indeed, one in which every set of $M$ becomes countable.
\end{corollary}

In slogan form, the corollary expresses what can be seen as a sweeping new multiverse modal principle: any statement true in an inner model is true again in an end-extension.

This corollary is an immediate consequence of Theorem~\ref{Theorem.Universal-sequence-for-suitable-theories} statement (4), but one can also give a direct proof of the same style as we gave in Section~\ref{Section.A-new-proof-of-the-Barwise-extension-theorem} for the Barwise extension theorem.

We are fascinated by some of the particular instances of this phenomenon. For example, if $M$ has a measurable cardinal $\kappa$, then we can collapse it by forcing, so that $M[G]$ thinks $\kappa$ is now a countable ordinal. But nevertheless, there will be an end-extension $N$ of $M[G]$ which again has a measurable cardinal---a new measurable cardinal is resurrected.

\begin{corollary}\label{Corollary.Applications-of-main-theorem}\ 
 \begin{enumerate}
  \item Every countable model $M$ of $\ZFC$ with a measurable cardinal has a covering end-extension to a model $N$ of $\ZFC+V=L[\mu]$.
  \item Every countable model $M$ of $\ZFC$ with extender-based large cardinals has a covering end-extension to a model $N$ satisfying $\ZFC+V=L[\vec E]$.
  \item Every countable model $M$ of $\ZFC$ with infinitely many Woodin cardinals and a measurable above has a covering end-extension to a model $N$ of $\ZF+\text{AD}+V=L(\mathbb{R})$.
  \item And in each case, we can furthermore arrange that every set of $M$ is countable in the extension model $N$. \qed
 \end{enumerate}
\end{corollary}

By Theorem~\ref{Theorem.Universal-sequence-for-suitable-theories}, we can of course also control the $\Sigma_1$-definable universal finite sequence in such end-extensions.

\section{End-extensional potentialism}\label{Section.Potentialism}

We should like now to draw out the consequences of the main theorem for set-theoretic end-extensional potentialism. Consider the potentialist system consisting of the countable models of set theory, a multiverse of models of set theory, forming a Kripke model under the end-extension relation, where here we wish to allow the trivial end-extension $M \subseteq M$. (Let us consider the models of a fixed suitable theory $T$, such as any computably enumerable theory $\ZFbar$ extending \ZF.) In this potentialist system we may naturally interpret the modal operators, so $\possible\varphi$ is true in a model $M$, if there is an end-extension in which $\varphi$ is true; and $\necessary\varphi$ is true in $M$, if all end-extensions satisfy $\varphi$. This is therefore a version of potentialism---one of many in set theory---as described in~\cite{HamkinsLinnebo:Modal-logic-of-set-theoretic-potentialism}. Let us call it \emph{end-extensional potentialism}.

Linnebo and others~\cite{Linnebo:2013-PHS, LinneboShapiro2017:Actual-and-potential-infinity, HamkinsLinnebo:Modal-logic-of-set-theoretic-potentialism} have emphasized the contrast between height-potentialism and width-potentialism in the philosophy of set theory, where with height-potentialism the set-theoretic universe fragments can grow taller, perhaps adding new ordinals on top, while with width-potentialism, the universe grows outward, perhaps adding new subsets to old sets as with forcing.

End-extensional potentialism is naturally a form both of height-potentialism and width-potentialism, since the extensions of a model can extend the ordinals of that model to taller ordinals and also new subsets of old sets can appear in an end-extension.

\begin{theorem}\label{Theorem.Hierarchy-end-extensional-potentialism-is-S4}
The modal logic of end-extensional potentialism, for the countable models of any fixed suitable set theory, is exactly \theoryf{S4}.
\end{theorem}

This result is proved exactly as the analogous results for arithmetic potentialism~\cite{Hamkins:The-modal-logic-of-arithmetic-potentialism} and rank-extensional potentialism~\cite{HamkinsWoodin:The-universal-finite-set} using, respectively, the universal algorithm and the $\Sigma_2$-definable universal finite set. Specifically, see \cite[theorem~28]{Hamkins:The-modal-logic-of-arithmetic-potentialism} for full details. The point is, any potentialist system which admits a universal finite sequence---a definition for a finite sequence which can be arbitrarily extended as you extend in the potentialist system---will have \theoryf{S4} as its modal validities.

Let us highlight, however, the important fact that getting \theoryf{S4} as an upper bound for the modal validities in general needs to allow formulae to have a single natural number parameter (for the length of the universal finite sequence in the current world). As we discuss in Section~\ref{Section.Maximality-Principle}, if you consider only sentences then the modal validities may go beyond \theoryf{S4}.

We also find it natural to consider a related and somewhat more permissive potentialist system, the \emph{$\Delta_0$-elementary potentialist system}, which consists of the countable models of set theory, ordered by $\Delta_0$-elementary extension, rather than just end-extension. That is, $N$ extends $M$ in this potentialist system if $M$ is a $\Delta_0$-elementary submodel of $N$. Since end-extensions are $\Delta_0$-elementary, this is somewhat more permissive than end-extensional potentialism.

\begin{corollary}
The modal logic of $\Delta_0$-elementary potentialism, for the countable models of any fixed suitable set theory, is exactly \theoryf{S4}.
\end{corollary}

\begin{proof}
The point is, the universal finite sequence for end-extensions is also a universal finite sequence for this potentialist system. The universal finite sequence for end-extensions gives us the needed $\Delta_0$-elementary extensions satisfying the new facts. And if $M$ is a $\Delta_0$-elementary submodel of $N$ then $N$ agrees with the $\Sigma_1$-theory of $M$, so the universal finite sequence as defined in $N$ must end-extend the sequence as defined in $M$.
\end{proof}

\section{The end-extensional maximality principle}\label{Section.Maximality-Principle}

In this section we introduce and investigate the end-extensional maximality principle, as well as the corresponding maximality principle for $\Delta_0$-elementary extensions of models of set theory. 

\begin{definition}\rm 
The \emph{maximality principle} is true at a model $M$ in a potentialist system, if $M\satisfies\possible\necessary\varphi\implies\varphi$ for any assertion $\varphi$ in the language of $M$. 
\end{definition}

Thus, the end-extensional maximality principle asserts that every possibly necessary statement is already true, in the potentialist system consisting of the countable models of set theory under end-extension, and similarly, the $\Delta_0$-elementary-extensional maximality principle makes this assertion with respect to $\Delta_0$-elementary extensions.

In these two cases, the main theorem of the article shows that we cannot allow parameters into the schema, even merely natural number parameters, because the assertion that the universal sequence is defined at $n$ is a possibly necessary statement that will not yet be true for large enough $n$. (Meanwhile, in various other potentialist systems, such as those considered in~\cite{HamkinsLinnebo:Modal-logic-of-set-theoretic-potentialism}, one can sometimes allow parameters into the schema, or certain kinds of parameters.) For this reason, it is also clear that no $\omega$-standard model of set theory can satisfy the maximality principle, since standard natural numbers are absolutely definable.

Let us begin by establishing that end-extensional and $\Delta_0$-elementary possibility for assertions in the language of set theory, in the context of $\omega$-nonstandard models, are equivalent to each other, and furthermore this kind of possibility is expressible as a first-order schema in the language of set theory. 

We found it worthwhile to aim for a somewhat general version of this result, and so let us say that a theory $T$ is \emph{very suitable}, if (i) $T$ is suitable; (ii) $T$ satisfies the Keisler-Morley theorem, meaning that every countable model of $T$ has a covering elementary end-extension; (iii) $T$ proves the Mostowski collapse lemma; and (iv) $T$ proves the reflection schema with countable parameters, namely, if $\varphi(a)$ holds of a hereditarily countable set $a$, then every set is contained in a transitive model of $\varphi(a)$.

This notion is stronger than the merely suitable theories, and perhaps the main interest is simply the theory \ZFC\ itself. In particular, it follows from well-known facts that \ZF\ is very suitable, as well as every computably enumerable extension of \ZF; and $\ZFm + V=L$ also is very suitable. On the other hand, work of Friedman, Gitman, and Kanovei~\cite{FriedmanGitmanKanovei2018:A-model-of-AC-not-DC} shows that $\ZFm$ is not very suitable.

\begin{theorem}[Characterization of end-extensional possibility]\label{Theorem.Characterization-of-LST-end-extension-possibility}
Consider a countable $\omega$-nonstandard model $\<M,\in^M>$ in the potentialist system of a fixed very suitable set theory $T$, such as \ZF\ or any c.e.~extension of \ZF. In particular $T$ is computably enumerable and accordingly let $T_k$ be the subtheory of $T$ consisting of the first $k$ axioms along a fixed enumeration. For any assertion $\varphi(a)$ in the language of set theory about a countable object $a \in M$, the following are equivalent:
\begin{enumerate}
  \item $M \satisfies \possible \varphi(a)$ in the end-extensional potentialist system of the countable models of $T$. That is, $M$ has an end-extension satisfying $T$ and $\varphi(a)$.
  \item $M \satisfies \possible \varphi(a)$ in the $\Delta_0$-elementary potentialist system. That is, $M$ has a $\Delta_0$-elementary extension satisfying $T$ and $\varphi(a)$.
  \item $M$ thinks of every countable transitive set $m \ni a$ and every standard number $k$ that the structure $\<m,\in^M>$ end-extends to a model $\<N,\in^N>\satisfies T_k+\varphi(a)$.
  \item $M$ thinks of every real $x$ and every standard natural number $k$ that there is an $\omega$-model of $T_k + \varphi(a)$ which contains $x$.
  \item (For sentences) $\varphi$ is consistent with $T$ plus the $\Sigma_1$-theory of $M$
\end{enumerate}
\end{theorem}

\begin{proof}
$(1 \Rightarrow 2)$ This is immediate, since every end-extension is $\Delta_0$-elementary.

$(2 \Rightarrow 3)$ Suppose toward a contradiction that $(2)$ holds and $(3)$ fails. Fix $k$ witnessing the failure and fix countable transitive $m \in M$ so that $M$ thinks $m$ has no end-extension to a model of $T_k + \varphi(a)$. This is a $\Pi^1_1$-fact about $m$ and $a$, so $M$ has a ranking function witnessing that the tree corresponding to the $\Pi^1_1$-fact is well-founded. This ranking function will remain a ranking function in any $\Delta_0$-elementary extension, so any such extension must agree that $m$ has no such end-extension. By $(2)$, let $N$ be a $\Delta_0$-elementary extension of $M$ in which $T+\varphi(a)$ holds. Since the Keisler--Morley theorem holds for $T$, we may assume without loss that $N$ is a covering extension, with $c \in N$ so that $N \models x \in c$ for each $x \in M$. By the very suitable reflection schema, $c$ has an end-extension in $N$ to a model of $T_k + \varphi(a)$. This end-extension of $c$ is also an end-extension of $m$, and so $N$ thinks that $m$ has an end-extension to a model of $T_k + \varphi(a)$, a contradiction.

$(3 \Rightarrow 1)$ By overspill there is a nonstandard $k$ so that $M$ thinks every countable transitive set end-extends to a model of $T_k + \varphi(a)$. By the Keisler--Morley theorem for $T$, let $M^+$ be a countable elementary end-extension of $M$ with a covering set $m$, and let $M^+[G]$ be a forcing extension which collapses $m$ to be countable. The assertion that every countable transitive set end-extends to a model of $T_k + \varphi(a)$ is $\Pi^1_2$, and thus by Shoenfield absoluteness we can transfer its truth in $M^+$ to $M^+[G]$. So $M^+[G]$ sees an end-extension of $m$ satisfying $T_k + \varphi(a)$. This end-extension witnesses that $M$ satisfies $\possible \varphi(a)$.

$(3 \Leftrightarrow 4)$ This is a standard fact, using the Mostowski collapse lemma to know that an arbitrary countable transitive set can be coded by a real. Note that this equivalence does not need $M$ to be $\omega$-nonstandard nor most of the assumptions on $T$.

$(1 \Rightarrow 5)$ The $\Sigma_1$-theory of a model is preserved by going to end-extensions, so if $\varphi$ holds in some end-extension of $M$ then it must be consistent with the $\Sigma_1$-theory of $M$.

$(5 \Rightarrow 3)$ Let $N$ be a model of $T + \varphi$ plus the $\Sigma_1$-theory of $M$. By the reflection principle satisfied by $T$ we get that $N$ thinks that every countable transitive set is contained in a model of $T_k + \varphi$, for any standard $k$. Suppose toward a contradiction that $(3)$ fails, so for some standard $k$, the model $M$ thinks there is a countable transitive set $m$ with no end-extension to a model of $T_k + \varphi$. As in the argument for $(1 \Rightarrow 2)$, this is a $\Pi^1_1$-fact about $m$ and so its truth in $M$ is witnessed by a certain well-founded tree with a ranking function to the ordinals. The assertion that there is such $m$ thus appears in the $\Sigma_1$-theory of $M$. So $N$ must think there is a countable transitive set with no end-extension to a model of $T_k + \varphi$, a contradiction.
\end{proof}

It follows as a corollary that if $M\satisfies\necessary\varphi(a)$, either with end-extensional or $\Delta_0$-elementary-extensional potentialism, then there is a concrete reason for this, namely, there is some standard $k$ and some countable transitive set $m\ni a$ in $M$ such that $M$ thinks $m$ has no end-extension to a model of $T_k+\neg\varphi(a)$. This is a $\Sigma^1_2$-assertion in the parameter $a$. 

It is natural to inquire whether the parameter $a$ of Theorem~\ref{Theorem.Characterization-of-LST-end-extension-possibility} could be uncountable in $M$. The answer is no. In general, neither $(1 \Rightarrow 3)$ nor $(2 \Rightarrow 3)$ holds if $a$ is uncountable in $M$. To see this, consider $a = \omega_1^M$. Then it is possible---in either potentialist system---that $a$ be made countable. However, no $n \in M$ which end-extends $a$ can think that $a$ is countable.

For the rest of this discussion, fix a very suitable theory $T$. We say that a collection $S$ of $\Sigma_1$-sentences is a \emph{maximal $\Sigma_1$-theory over $T$} if $T + S$ is consistent and $S$ is maximal among the sets of $\Sigma_1$-sentences consistent with $T$. One can easily construct a maximal $\Sigma_1$-theory by enumerating the $\Sigma_1$-sentences and adding them one at a time, so long as the resulting theory is consistent with $T$. The same argument shows that every collection of $\Sigma_1$-sentences consistent with $T$ is contained in a maximal $\Sigma_1$-theory.

An $\omega$-standard model $M$ of $T$, of course, can never have a maximal $\Sigma_1$-theory, because assertions like ``the $\Sigma_1$-definable universal finite sequence has length at least $n$'' are $\Sigma_1$-sentences for standard $n$. For the same reason, the universal finite sequence in any model of $T$ with a maximal $\Sigma_1$-theory will necessarily have nonstandard length.

\begin{theorem}\label{Theorem.End-extensional-maximality-principle}
The following are equivalent for any very suitable theory $T$ and any countable $M \satisfies T$. 
\begin{enumerate}
\item $M$ satisfies the end-extensional maximality principle.
\item $M$ satisfies the $\Delta_0$-elementary-extensional maximality principle.
\item The $\Sigma_1$-theory of $M$ is a maximal $\Sigma_1$-theory over $T$.
\end{enumerate}
\end{theorem}

\begin{proof}
$(1 \Leftrightarrow 2)$ This holds because $M$ satisfies $\possible \varphi$ in one potentialist system if and only if it does in the other, by Theorem~\ref{Theorem.Characterization-of-LST-end-extension-possibility}.

$(1 \Rightarrow 3)$ Suppose $\varphi$ is a $\Sigma_1$-sentence consistent with the $\Sigma_1$-theory of $M$. By theorem~\ref{Theorem.Characterization-of-LST-end-extension-possibility} we get that $\possible \varphi$ holds at $M$. Because $\Sigma_1$-assertions must be necessary, $\possible \necessary \varphi$ holds at $M$, and so by the end-extensional maximality principle we get that $\varphi$ holds at $M$.

$(3 \Rightarrow 1)$ Suppose that $\possible \necessary \varphi$ holds at $M$, so there is an end-extension $N \satisfies T+\necessary \varphi$. By theorem~\ref{Theorem.Characterization-of-LST-end-extension-possibility}, there is a standard natural number $k$ so that $N$ thinks there is a real $x$ which is not contained in any $\omega$-model of $T_k + \neg\varphi$. The assertion that there is no model of $T_k + \neg\varphi$ containing $x$ is $\Pi^1_1$ in the parameter $x$, and so its truth in $N$ is witnessed by the existence of a certain canonically-associated well-founded tree. This is equivalent to a $\Sigma_1$-assertion. Because the $\Sigma_1$-theory of $M$ is maximal we therefore get that this is already true in $M$. So $\necessary \varphi$ holds at $M$, and in particular $\varphi$ holds at $M$.
\end{proof}

\begin{corollary}
Every countable model of $T$ has a $\Delta_0$-elementary extension---not necessarily an end-extension---to a model satisfying the end-extensional maximality principle, and therefore also the $\Delta_0$-elementary-extensional maximality principle.
\end{corollary}

\begin{proof}
Let $M$ be a countable model of $T$ and let $S_0$ be $T$ plus the $\Delta_0$-elementary diagram of $M$. Extend $S_0$ to a maximal $\Sigma_1$-theory $S$. Then $S$ has a countable model, which must be a $\Delta_0$-elementary extension of $M$. And by theorem~\ref{Theorem.End-extensional-maximality-principle} any countable model of $S$ must satisfy the two maximality principles.
\end{proof}

If $M$ is $\omega$-standard, then this extension cannot be an end-extension, as any end-extension of $M$ must also be $\omega$-standard, and hence cannot have a maximal $\Sigma_1$-theory.

\begin{question}
Does every countable $\omega$-nonstandard model of \ZF\ end-extend to a model satisfying the end-extensional maximality principle? By theorem~\ref{Theorem.End-extensional-maximality-principle} this is equivalent to asking whether every countable $\omega$-nonstandard model of \ZF\ end-extends to a model with a maximal $\Sigma_1$-theory.
\end{question}

The corresponding question is also open for models of arithmetic, as discussed in \cite{Hamkins:The-modal-logic-of-arithmetic-potentialism}.


\bibliographystyle{alpha}
\bibliography{HamkinsBiblio,MathBiblio,WebPosts,local}

\begin{thebibliography}{Ham18b}

\bibitem[Bar71]{Barwise1971:InfinitaryMethodsInTheModelTheoryOfSetTheoryLC69}
Jon Barwise.
\newblock Infinitary methods in the model theory of set theory.
\newblock In {\em Logic {C}olloquium '69 ({P}roc. {S}ummer {S}chool and
  {C}olloq., {M}anchester, 1969)}, pages 53--66. North-Holland, Amsterdam,
  1971.

\bibitem[BE17]{BlanckEnayat2017:Marginalia-on-a-theorem-of-Woodin}
Rasmus Blanck and Ali Enayat.
\newblock Marginalia on a theorem of {W}oodin.
\newblock {\em J. Symb. Log.}, 82(1):359--374, 2017.

\bibitem[Bla17]{Blanck2017:Dissertation:Contributions-to-the-metamathematics-of-arithmetic}
Rasmus Blanck.
\newblock {\em Contributions to the Metamathematics of Arithmetic}.
\newblock PhD thesis, University of Gothenburg, 2017.

\bibitem[FKG19]{FriedmanGitmanKanovei2018:A-model-of-AC-not-DC}
Sy-David Friedman, Vladimir Kanovei, and Victoria Gitman.
\newblock A model of second-order arithmetic satisfying ac but not dc.
\newblock {\em Journal of Mathematical Logic}, 19(1):1850013, 2019.

\bibitem[GHJ16]{GitmanHamkinsJohnstone2016:WhatIsTheTheoryZFC-Powerset?}
Victoria Gitman, Joel~David Hamkins, and Thomas~A. Johnstone.
\newblock What is the theory {ZFC} without {Powerset}?
\newblock {\em Math.~Logic Q.}, 62(4--5):391--406, 2016.

\bibitem[Ham14]{Hamkins2014:MultiverseOnVeqL}
Joel~David Hamkins.
\newblock A multiverse perspective on the axiom of constructibility.
\newblock In {\em Infinity and Truth}, volume~25 of {\em LNS Math Natl. Univ.
  Singap.}, pages 25--45. World Sci. Publ., Hackensack, NJ, 2014.

\bibitem[Ham18a]{Hamkins:The-modal-logic-of-arithmetic-potentialism}
Joel~David Hamkins.
\newblock The modal logic of arithmetic potentialism and the universal
  algorithm.
\newblock {\em ArXiv e-prints}, pages 1--35, 2018.
\newblock under review.

\bibitem[Ham18b]{Hamkins.blog2018:A-new-proof-of-the-barwise-extension-theorem}
Joel~David Hamkins.
\newblock A new proof of the barwise extension theorem, without infinitary
  logic.
\newblock {Mathematics and Philosophy of the Infinite}, 2018.

\bibitem[HL18]{HamkinsLinnebo:Modal-logic-of-set-theoretic-potentialism}
Joel~David Hamkins and \O{}ystein Linnebo.
\newblock The modal logic of set-theoretic potentialism and the potentialist
  maximality principles.
\newblock {\em to appear in Review of Symbolic Logic}, 2018.

\bibitem[HW17]{HamkinsWoodin:The-universal-finite-set}
Joel~David Hamkins and W.~Hugh Woodin.
\newblock The universal finite set.
\newblock {\em ArXiv e-prints}, pages 1--16, 2017.
\newblock manuscript under review.

\bibitem[Kau81]{Kaufmann1981.On-existence-of-sigma_n-end-extensions}
Matt Kaufmann.
\newblock On existence of {$\Sigma_n$} end extensions.
\newblock In {\em Logic Year 1979--80 (Proc. Seminars and Conf. Math. Logic,
  Univ. Connecticut, Conn., 1979/90)}, volume 859 of {\em Lecture Notes in
  Mathematics}, pages 92--103. Springer, Berlin, 1981.

\bibitem[KM68]{KeislerMorley1968:ElementaryExtensionsOfModelsOfSetTheory}
H.~Jerome Keisler and Michael Morley.
\newblock Elementary extensions of models of set theory.
\newblock {\em Israel J. Math.}, 6:49--65, 1968.

\bibitem[Lin13]{Linnebo:2013-PHS}
{\O}ystein Linnebo.
\newblock The potential hierarchy of sets.
\newblock {\em Review of Symbolic Logic}, 6(2):205--228, 2013.

\bibitem[LS17]{LinneboShapiro2017:Actual-and-potential-infinity}
{\O}ystein Linnebo and Stewart Shapiro.
\newblock Actual and potential infinity.
\newblock {\em No{\^u}s}, 2017.

\bibitem[Sim09]{Simpson2009:SubsystemsOfSecondOrderArithmetic}
Stephen~G. Simpson.
\newblock {\em Subsystems of second order arithmetic}.
\newblock Perspectives in Logic. Cambridge Univ. Press; Association for
  Symbolic Logic, second edition, 2009.

\bibitem[Woo11]{Woodin2011:A-potential-subtlety-concerning-the-distinction-between-determinism-and-nondeterminism}
W.~Hugh Woodin.
\newblock A potential subtlety concerning the distinction between determinism
  and nondeterminism.
\newblock In {\em Infinity}, pages 119--129. Cambridge Univ. Press, Cambridge,
  2011.

\bibitem[Zar82]{Zarach1982:Unions-of-ZFm-models}
Andrzej Zarach.
\newblock Unions of {$\mathrm{ZF}^-$}-models which are themselves
  {$\mathrm{ZF}^-$}-models.
\newblock In D.~{Van Dalen}, D.~Lascar, and T.J. Smiley, editors, {\em Logic
  Colloquium '80}, volume 108 of {\em Studies in Logic and the Foundations of
  Mathematics}, pages 315 -- 342. Elsevier, 1982.

\bibitem[Zar96]{Zarach1996:ReplacmentDoesNotImplyCollection}
Andrzej~M. Zarach.
\newblock Replacement {$\nrightarrow$} collection.
\newblock In {\em G\"odel '96 ({B}rno, 1996)}, volume~6 of {\em Lecture Notes
  Logic}, pages 307--322. Springer, Berlin, 1996.

\end{thebibliography}

\end{document}